\documentclass[12pt,thmsa, reqno]{amsart}

\usepackage{amssymb, euscript,  mathrsfs}
\usepackage[dvips]{graphics}
\usepackage[cmtip,all]{xy}

\input{amssym.def}
\input{amssym.tex}

\usepackage{lineno}

\def\Cal{\mathcal}

\def\C{{\Cal C}}

\def\S{{\Cal S}}

\def\vnk{V_{n,k}}

\def\bbr{{\Bbb R}}

\def\bbc{{\Bbb C}}

\def\bbs{{\Bbb S}}

\def\grad{{\hbox{\rm grad}}}

\def\vnk{V_{n,k}}

\def\rn{\bbr^n}
\def\sn{S^{n-1}}

\def\part{\partial}
\def\intl{\int\limits}

\def\Gam{\Gamma}

\def\a{\alpha}
\def\om{\omega}

\def\Del{\Delta}
\def\del{\delta}
\def\vp{\varphi}

\def\gam{\gamma}

\def\lam{\lambda}

\def\e{\varepsilon}

\def\chi{{\bf 1}}

\def\snm1{\bbs^{n-1}}

\def\Re{\mathrm{Re}\,}

\def\intl{\int\limits}

\def\grad{\mathrm{grad}}

\def\vnk{\mathrm{V}_{n,k}}

\def\Cs{\mathscr{C}}

\def\Cs{\mathscr{C c 1234}}

\def\cd{\stackrel{*}{\C}\!{}_{m, k}^\lam}
\def\sd{\stackrel{*}{\S}\!{}_{m, k}^\lam}
\def\cd0{\stackrel{*}{\C}\!{}_{m, k}^\lam}
\def\sd0{\stackrel{*}{\S}\!{}_{m, k}^\lam}

\def\ncd0{\stackrel{*}{\Cs}\!{}_{m, k}^\lam}

\newtheorem{theorem}{Theorem}[section]
\newtheorem{lemma}[theorem]{Lemma}

\theoremstyle{definition}
\newtheorem{definition}[theorem]{Definition}

\theoremstyle{remark}
\newtheorem{remark}[theorem]{Remark}

\numberwithin{equation}{section}

\theoremstyle{corollary}

\newtheorem{proposition}[theorem]{Proposition}

\numberwithin{equation}{section}

\newcommand{\be}{\begin{equation}}
\newcommand{\ee}{\end{equation}}

\newcommand{\bea}{\begin{eqnarray}}
\newcommand{\eea}{\end{eqnarray}}
\newcommand{\Bea}{\begin{eqnarray*}}
\newcommand{\Eea}{\end{eqnarray*}}

\def\sideremark#1{\ifvmode\leavevmode\fi\vadjust{\vbox to0pt{\vss
 \hbox to 0pt{\hskip\hsize\hskip1em
\vbox{\hsize2cm\tiny\raggedright\pretolerance10000
 \noindent #1\hfill}\hss}\vbox to8pt{\vfil}\vss}}}%

\begin{document}

\title[The Fourier Transform Approach  ]{The Fourier Transform Approach to Inversion of  $\lam$-Cosine and Funk Transforms on the Unit Sphere}

\author{ B. Rubin}

\address{Department of Mathematics, Louisiana State University, Baton Rouge,
Louisiana 70803, USA}
\email{borisr@lsu.edu}




\begin{abstract} We use the  classical Fourier analysis on $\rn$ to introduce analytic families of weighted differential operators on the unit sphere. These operators are polynomial functions of the usual Beltrami-Laplace operator. New inversion formulas are obtained for  totally geodesic Funk transforms on the sphere and the relevant  $\lam$-cosine transforms.
 \end{abstract}

\maketitle

\section{Introduction}


\setcounter{equation}{0}
\noindent
The $\lam$-cosine transform
\be\label{af000}    (C^\lam
f)(u)=   \intl_{\sn}  f(v)
|u \cdot v|^{\lam} \,dv, \qquad u\in \sn, \ee
on the unit sphere $\sn$ in $\rn$, $n \ge 3$,    arises in different branches  of mathematics. The terminology for $\lam =1$ amounts to Lutwak \cite [p. 385]{Lu} in convex geometry.
There exist many modifications and generalizations of this operator; see, e.g.,  \cite {OPR} and references therein.
   The associated operator
   \be\label{af000r}    (Ff)(u)=   \intl_{\{v\in \sn :\,  u \cdot v=0\}} \!\!\!\! f(v)
 \,d_u v, \ee
where $d_u v$ is the relevant probability measure, is called the Funk transform; cf. \cite{Fu11, Fu13} for $n=3$. In convex geometry, operators   (\ref{af000}) are also known as the $\a$-cosine transforms or $p$-cosine transforms. Operators (\ref{af000r}) are sometimes called spherical Radon transforms or Minkowskli-Funk transforms. These operators play an important role in the study of projections and sections of convex bodies; see, e.g.,  \cite {Ga, Ko, Zha}, to mention a few.

A variety of diverse inversion formulas for the operators (\ref{af000}) and (\ref{af000r}) can be found in the literature; see, e.g., \cite {GGG, H11, P, Ru15}, and references therein. Among them, Helgason's polynomial type inversion formula for $F$
is probably the most beautiful. For example, if $n$ is even, then an even smooth  function $f$ can be reconstructed from $\vp= Ff$ as
\be\label {mjbt} f= P(\Del_S)F \vp,\ee
where $P(\Del_S)$ is a polynomial of the Beltrami-Laplace operator $\Del_S$. This result was obtained by Helgason \cite[p. 284]{H59} in the more general setting  for totally geodesic submanifolds  of constant curvature spaces of arbitrary dimension; see also \cite[p. 133]{H11}. Because of the inevitable dimensionality restriction, (\ref{mjbt}) has a local nature.
 The corresponding non-local 
 polynomial type inversion formulas, which include inversion of (\ref{af000r})
  for $n$ odd, were obtained by the author in \cite [Theorem 1.2]{Ru02}; see also \cite [Section 5.1.6] {Ru15}.
 The main tool in the proof of these formulas  is spherical harmonic decomposition.

\vskip 0.2 truecm

\noindent {\bf The aim of the paper and motivation.} We plan  to obtain polynomial type inversion formulas  for (\ref{af000r}) and (\ref{af000})  without using spherical harmonics and in a more compact form. Specifically, instead of a  polynomial of the Beltrami-Laplace operator we introduce one ``weighted'' differential operator on $\sn$. The latter can be  defined by making use of homogeneous extension onto $\rn \setminus \{0\}$ with subsequent multiplication by the weight function, like $|x|^\lam$, and implementation of the Fourier transform technique.

  The suggested approach is motivated by our conjecture that similar inversion problems for  $\lam$-cosine and Funk type transforms on the Stiefel and Grassmann manifolds can be treated using  the standard Fourier Analysis on the ambient  matrix space. Some results in this direction are obtained in
 \cite {Ru13}. We plan to address  this topic in   another publication.

\vskip 0.2 truecm

\noindent {\bf Plan of the paper.}  Section 2 deals  with the Funk transform (\ref{af000r}) and the relevant $\lam$-cosine transforms. The latter are suitably normalized. We introduce weighted differential operators on $\sn$ and use them to obtain new  inversion formulas for our transforms. The Fourier transform technique in this section  amounts  apparently to Semyanistyi \cite{Se}; see also \cite {Ko, Sa83, Sa83a} and \cite [Lemma A.92]{Ru15}.  Section 3 contains basic facts about the corresponding $\lam$-sine transforms. In Section 4 we extend inversion formulas from Section 2 to the case of lower-dimensional totally geodesic Funk transforms and the relevant $\lam$-cosine transforms.

\section{ $\lam$-Cosine Transforms, Funk Transforms, and Weighted Beltrami-Laplace Operators} \label{sphere1}

\noindent  The main reference for this section is \cite [Sections 5.1, A.13]{Ru15}. For the sake of convenience, we normalize the integral (\ref {af000}) and denote
 \be\label{af}    (\Cs^\lam
f)(u)= \gam (\lam)   \intl_{\sn}  f(v)
|u \cdot v|^{\lam} \,d_*v,\qquad u \in
\sn,  \ee
where $d_*v$ stands for the standard $O(n)$-invariant probability measure on $\sn$,
\be\label{beren}
\gam (\lam)\!=\!\frac{\pi^{1/2}\,\Gamma( -\lam/2)}{\Gamma (n/2)\, \Gamma ((\lam +1)/2)}, \qquad \Re \, \lam \!>\!-1, \quad \lam
\!\neq \!0,2,4
, \ldots .\ee

It is  assumed that  $f \in C^\infty (\sn)$, though many results  extend to wider classes of functions. For technical reasons we prefer the following definition of the space $C^\infty (\sn)$, which, however, is equivalent to the standard one via atlas on $\sn$; see, e.g., \cite [Proposition 1.29]{Ru15}.

\begin{definition} We say that $f \in C^\infty (\sn)$ if the extended function
\be\label{dlim1}
 \tilde f(x)=f(x/|x|), \qquad  x \in \rn \setminus \{0\},\ee
 belongs to $C^\infty  (\rn \setminus \{0\})$ in the usual sense.
\end {definition}

We  assume $f$ to be even, i.e.,  $f \in C^\infty_{even} (\sn)$, because odd functions form the kernel (null space) of $\Cs^\lam$.
If $f \in C^\infty_{even} (\sn)$, then $\Cs^\lam f$ extends to all $\lam \in \bbc$  meromorphically
 with the  poles $ \lam \!= \!0,2,4, \ldots $, so that $a.c. \Cs^\lam f$ (the analytic continuation of $\Cs^\lam f$) belongs to $C^\infty_{even} (\sn)$.
   The limit case
 $\lam=-1$ represents a constant multiple  of  the  Funk transform (\ref{af000r}):
\be\label{lim1}
\Cs^{-1} f\equiv \lim\limits_{\lam \to -1} \Cs^\lam f\!= \!c_n\, Ff,\qquad c_n\!=\!\frac{\pi^{1/2}}{\Gamma ((n\!-\!1)/2)}.
\ee

\subsection{From the Euclidean Fourier Analysis to the Spherical Cosine Transform}\label {Fourier}

  Let $\Delta_S$ be the Beltrami-Laplace operator on $\sn$, that can be  defined  by the formula
  \[(\Delta_S f) (x/|x|)=  |x|^{2}(\Del \tilde f)(x), \quad  \quad  x \in \rn \setminus \{0\}.\]
  Here $\tilde f$ is the extended function  (\ref{dlim1}) and $\Del$ is  the usual Laplacian or $\rn$. The notation $I$ will be used for the identity operator. For $f \in C^\infty (\sn)$, the  expression $\Cs^\lam f$ is understood (if necessary) as the meromorphic continuation of the absolutely convergent integral  (\ref{af}).

 \begin{proposition}\label {liut0}  If $f \in C^\infty_{even} (\sn)$, $\lam \in \bbc \setminus \{ -2, 0,2,4, \ldots\}$, then
 \be\label{FueedvbcAAA} -\frac{1}{4}\,[\Delta_S+ (\lam +2)(n+\lam) I]\,  \Cs^{\lam +2} f =\Cs^\lam f.  \ee
  \end{proposition}
\begin{proof}
 This statement   was  proved in \cite [p. 285]{Ru15} using the theory of spherical harmonics. Below we suggest a simple proof, which does not need any knowledge of spherical harmonics and  relies on the Fourier transform technique exclusively.

  Let $S(\bbr^n)$ be the Schwartz space
  of rapidly decreasing smooth functions $\om$ on $\rn$, and let
 $ \hat\om(y)=\int_{\bbr^n} \om (x) e^{i x\cdot y}dx$
  be the Fourier transform  of $\om$. Given a function $f$ on $\sn$, we denote by
 \be\label {lase} (E_\lam f)(x)=|x|^{\lam} f(x/|x|), \qquad x \in \rn \setminus \{0\},\ee
    the $\lam$-homogeneous extension of $f$. Then for all complex $\lam \neq \!0,2,4, \ldots $,
  \be\label{eq5}
 \left(E_{\lam}\, \Cs^\lam f, \hat\om\right)\!=\!c_\lam \, \left (E_{-\lam-n}  f,  \om\right),   \qquad c_\lam \!=\!2^{n+\lam}\, \pi^{n/2},\ee
 where   both sides  are understood in the sense of analytic continuation;  see, e.g., \cite [pp. 526, 527]{Ru15}.
 This formula  gives the Fourier transform of the $S'$-distribution $E_{-\lam-n}  f$.
 Setting $\om_1 (x)=|x|^{2} \om (x)$ and
using (\ref{eq5}) repeatedly, we obtain
\bea
\left (\Del [E_{\lam +2} \Cs^{\lam +2} f], \hat\om\right)&=& -\left (E_{\lam +2} \Cs^{\lam +2} f, \hat\om_1 \right)\nonumber\\
&=& -c_{\lam +2} \, \left (E_{-\lam-2-n}  f,  \om_1\right)\nonumber\\&=&  -c_{\lam +2} \,\left (E_{-\lam-n}  f,  \om\right)\nonumber\\
&=&-\frac{c_{\lam +2}}{c_{\lam}} \left(E_{\lam}\, \Cs^\lam f, \hat\om\right)\nonumber\\
\label {abb}&=& -4\left(E_{\lam}\, \Cs^\lam f, \hat\om\right), \eea
 that is,  $\Del [E_{\lam +2} \Cs^{\lam +2} f]= -4 \,E_{\lam}\, \Cs^\lam f$ in the $S'$-sense.  Because both sides of the last equality  are smooth away from the origin, the pointwise equality follows:
\be\label {abbc}
 \Del \,[ |x|^{\lam +2} (\Cs^{\lam +2} f) (\tilde x)]=-4\,  |x|^{\lam} (\Cs^{\lam} f) (\tilde x),  \qquad  \tilde x=\frac{x}{|x|}. \ee
   Now, using the product formula
\be\label {lyt} \Del (\vp \psi)= \vp\Del \psi +2\,(\grad \,\vp )\cdot (\grad \,\psi) +\psi\Del \vp\ee
with $\vp (x)= |x|^{\lam +2}$ and $\psi (x)= (\Cs^{\lam +2} f) (\tilde x)$, and setting $|x|=1$, after simple calculations  we arrive at (\ref{FueedvbcAAA}); cf.  \cite [p. 491]{Ru15}.
\end{proof}

Let us focus on the ``Fourier part'' of the above reasoning and drop the second part, related to (\ref{lyt}). Denote
    \be\label {lase1}
(\Delta_{\lam} f) (u)= -\frac{1}{4}  (\Del E_{\lam +2} f)(x)\Big |_{x=u}, \qquad u \in \sn. \ee
Then (\ref{abbc}) yields an alternative version of Proposition \ref{liut0}:
 \be\label {lase2}
\Delta_{\lam} \Cs^{\lam +2} f=\Cs^{\lam} f,  \qquad \lam \in \bbc \setminus \{-2, 0,2,4, \ldots\}.\ee

More generally, let us replace the Laplace operator $\Del$ by its integer power $\Del^\ell$, $\ell \ge 0$. Then (\ref{abb}) becomes
\be\label {abbv}
\left (\Del^\ell [E_{\lam +2\ell} \Cs^{\lam +2\ell} f], \hat\om\right)= (-4)^\ell\, \left(E_{\lam}\, \Cs^\lam f, \hat\om\right),\ee
and, as above,
\be\label {abbc1}
 (\Del^\ell E_{\lam +2\ell} \Cs^{\lam +2\ell} f) (x)=(-4)^\ell\,  (E_{\lam}\, \Cs^\lam f)(x), \qquad x\neq 0. \ee
Denote
\be\label {lase1h}
(\Delta_{\lam, \ell} f) (u)= \left (-\frac{1}{4}\right )^\ell  (\Del^\ell E_{\lam +2\ell} f)(x)\Big |_{x=u}, \qquad u \in \sn. \ee
As a result, we obtain the following generalization of (\ref{lase2}).

 \begin{proposition}\label {liut0b}  If $f \in C^\infty_{even} (\sn)$, then
 \be\label {lase2h}
\Delta_{\lam, \ell} \Cs^{\lam +2\ell} f=\Cs^{\lam} f \ee
  for all complex $\lam$ satisfying $ \lam +2\ell\neq 0,2,4, \ldots$.
 \end{proposition}

The new operator $\Delta_{\lam, \ell}$ can be regarded as
 a ``weighted'' differential operator on $\sn$, thanks to the power weight $|x|^\lam$ in (\ref{lase}). Clearly,  $\Delta_{\lam, \ell}$ is a polynomial function of the Beltrami-Laplace operator $\Delta_S$.

\subsection{The Logarithmic  Cosine Transform}

We will also need the  logarithmic analogue of the  cosine transform (\ref{af}):
\be\label{zaqcvr6}
(\Cs_{log}f)(u)= \frac{2}{\Gam(n/2)}\intl_{\sn} f(v) \log {1\over |u \cdot v|}\, d_*v. \ee
 If
$\int_{\sn} f(v)\, d_*v=0$,  then  \cite[p. 528]{Ru15}
\be\label{zaqcvr8}
\lim\limits_{\lam\to 0} \Cs^\lam f = \Cs_{log}f. \ee

The logarithmic  cosine transform can be used to extend (\ref{lase2}) to the excluded values of $\lam$. For our purposes, it suffices to consider $\lam=-2$, when $\Cs^{-2}f$ (the analytic continuation of $\Cs^{\lam}f$  at $\lam =-2$) is well defined, but $\Cs^{0}f$ is, in general,  not defined.

\begin{lemma} \label{zret8n}
If $f\in C^\infty_{even} (\sn)$ and  $\int_{\sn} f(v)\, d_*v=0$,  then
 \be\label {lase2s}
\Delta_{-2} \Cs_{log}f  =\Cs^{-2} f.\ee
\end{lemma}
\begin{proof}   We denote by $C_c^\infty (\rn \setminus \{0\})$ the space of all $C^\infty $ functions on $\rn$ with compact support away from the origin.
Let  $\e=\lam +2$ and assume that $\e$ is a sufficiently small positive number. Then for any function $\vp \in C_c^\infty (\rn \setminus \{0\})$, 
\bea
&&\Big (\Del [E_{\lam +2} \Cs^{\lam +2} f], \vp \Big)=\Big (  (\Cs^{\e} f)(x/|x|), |x|^\e (\Del\vp)(x) \Big)\nonumber\\
&&=\left (\e  \gam (\e)   \intl_{\sn}  f(v)\, \frac{|\frac{x}{|x|} \cdot v|^\e -1}{\e}\, d_*v, \,|x|^\e (\Del\vp)(x)\right).\nonumber\eea
Passing to the limit, we obtain
\[
\lim\limits_{\lam \to -2} \Big (\Del [E_{\lam +2} \Cs^{\lam +2} f], \vp \Big)=  \Big ((\Cs_{log}f)(x/|x|), (\Del\vp)(x)\Big)=   (\Del E_0\Cs_{log}f, \vp).\]
Hence, by (\ref{abb}),
\[
\left (\Del E_0\Cs_{log}f, \vp\right)=-4 \lim\limits_{\lam \to -2}  \,\left(E_{\lam}\, \Cs^\lam f, \vp \right)=-4 \,\left(E_{-2}\, \Cs^{-2}f, \vp \right).\]
Here we recall that for any $f\in C^\infty_{even} (\sn)$ and $\vp \in C_c^\infty (\rn \setminus \{0\}) $,  the function
 \be\label {oiy} \lam \to \left(E_{-2}\, \Cs^{-2}f, \vp \right)\ee
 is meromorphic with the poles $0, 2, 4 \ldots$ and the
 analytic continuation of $\Cs^\lam f$ at any $\lam \notin \{ 0, 2, 4,\ldots\}$ belongs to $C^\infty_{even} (\sn)$.
Now (\ref{oiy}) yields a pointwise equality
\[
(\Del E_0\Cs_{log}f)(x)=-4 \,(E_{-2}\, \Cs^{-2}f)(x), \qquad x \neq 0.\]
Taking restriction to $x\in \sn$, we obtain (\ref{lase2s}).
\end{proof}

\begin{remark}  One can also extend (\ref{lase2h}) to the excluded value $\lam =-2\ell$ provided $\int_{\sn} f(v)\, d_*v=0$. The result is
 \be\label {lase2hlo}
\Delta_{-2\ell, \ell}\,\Cs_{log}f=\Cs^{-2\ell} f,  \qquad \ell= 1,2,3, \ldots.\ee
To justify (\ref{lase2hlo}), we proceed as in the proof of Lemma \ref{zret8n} and get
\[
\lim\limits_{\lam \to -2\ell} (\Del^\ell [E_{\lam +2\ell} \Cs^{\lam +2\ell} f], \vp )=   (\Del^\ell E_0\Cs_{log}f, \vp), \quad \vp \in C_c^\infty (\rn \setminus \{0\}).\]
Because
$\Del^\ell E_{\lam +2\ell} \Cs^{\lam +2\ell} f=(-4)^\ell\,  E_{\lam}\, \Cs^\lam f$ (see (\ref{abbc1})), we continue:
\[
 (\Del^\ell E_0\Cs_{log}f, \vp)=(-4)^\ell\,\lim\limits_{\lam \to -2\ell}  (E_{\lam}\, \Cs^\lam f, \vp)=(-4)^\ell (E_{-2\ell}\, \Cs^{-2\ell}f, \vp ).\]
The latter gives  a pointwise equality
\[
(\Del^\ell E_0\Cs_{log}f)(x)=(-4)^\ell\,(E_{-2\ell}\, \Cs^{-2\ell}f)(x), \qquad x \neq 0,\]
which implies (\ref{lase2hlo}).
\end{remark}

\subsection{Inversion Formulas}

 The  point of departure is the equality
\be\label{lkut}
  \Cs^{-\lam -n} \Cs^\lam f =f \ee
that holds for  $f\in C^\infty_{even} (\sn)$ and any complex $\lam$ satisfying
\[\lam, -\lam -n \neq 0,2,4, \ldots .\]
The expression on the left-hand side is understood in the sense of analytic continuation. The formula (\ref{lkut})  is transparent on spherical harmonics
 \cite [p. 284]{Ru15}, however, it can also be proved using
 the Fourier transform technique; cf. \cite[Theorem 7.7]{Ru13}.

Another ingredient of our algorithm is  the equality
 (\ref{lase2h}) that can be written as
  \be\label{lbkut}  \Cs^{\lam} f =
\Delta_{\lam, \ell} \Cs^{\lam +2\ell} f,  \qquad \lam +2\ell\neq 0,2,4, \ldots \, .  \ee
Here and on, $\Delta_{\lam, \ell}$ is the differential operator (\ref{lase1h}) that will be used  with different $\lam$ and $\ell$.

 Plugging (\ref{lbkut}) in (\ref{lkut}), we obtain
  \be\label{lqbkut}
 f=\Cs^{-\lam -n}\Delta_{\lam, \ell} \Cs^{\lam +2\ell} f,  \qquad  -\lam -n, \,\lam +2\ell\neq 0,2,4, \ldots \, .  \ee
This is our first inversion formula, in which the differential operator is located between two cosine transforms.
 Further, replacing $\lam$ by $-\lam -n$ in (\ref{lbkut}), we have
   \[  \Cs^{-\lam -n} f =
\Delta_{-\lam -n, \ell} \Cs^{-\lam -n +2\ell} f,  \qquad -\lam -n +2\ell\neq 0,2,4, \ldots \, .  \]
 If we replace $f$ by $ \Cs^\lam f$ in the last equality and make use of (\ref{lkut}), we obtain an alternative inversion formula, in which the differential operator is outside:
 \be\label{lkyyut}
 f = \Delta_{-\lam -n, \ell} \Cs^{-\lam -n +2\ell} \Cs^\lam f, \qquad \lam,  -\lam -n +2\ell\neq 0,2,4, \ldots \, .  \ee
 Both  formulas (\ref{lqbkut}) and (\ref{lkyyut}) can be specified for our purposes. Below we give some examples.

 \begin{theorem} \label{jhb67} {\rm (Inversion of the Funk transform)}. Let  $\vp =Ff$,
  \[f\in C^\infty_{even}(S^{n-1}), \qquad n\ge 3, \qquad c_n=\frac{\pi^{1/2}}{\Gamma ((n-1)/2)}.\]

\noindent {\rm (i)} If $n$ is even, $D=(c_n)^2 \Delta_{1-n, (n-2)/2}$, then
\be\label{ltqbkutm}
 f= FD\vp \quad \text {or}\quad f=D F\vp.\ee

\noindent {\rm (ii)} If $n$ is odd, $\tilde D= c_n\Delta_{1 -n, (n-1)/2 } $, then
\be\label{lkut4am} f=\vp_0 +\tilde D\Cs_{log} \vp, \qquad \vp_0=\intl_{\sn} \vp(u)\, d_*u.\ee
 \end{theorem}
  \begin{proof} (i) We first note that  $\Cs^{-1} f= c_n Ff= c_n \vp$; see (\ref{lim1}). Setting $\lam +2\ell=-1$ in (\ref{lqbkut}), we have
 \[
 f=c_n \, \Cs^{2\ell +1 -n}\Delta_{-2\ell -1, \ell} \,\vp,  \qquad  2\ell+1-n\neq 0,2,4, \ldots \, .  \]
 Choosing  $\ell=(n-2)/2$, we obtain
  \be\label{ltqbkut}
 f=c_n \, \Cs^{-1}\Delta_{1-n, (n-2)/2} \,\vp=\Cs^{-1} D\vp,\ee
 which gives the first equality in (\ref{ltqbkutm}).
The second equality follows from (\ref{lkyyut}) if we set $\lam =-1$, $ \ell=(n-2)/2$.

(ii) We have
 $\vp-\vp_0= F[f-\vp_0]$. Hence, by  (\ref{lase2hlo}),
\be\label{selkut3}
f-\vp_0=  c_n\Cs^{1 -n} [\vp-\vp_0]=c_n\Delta_{1 -n, (n-1)/2 }  \Cs_{log} [\vp-\vp_0].\ee
This gives $f=\vp_0 +c_n\Delta_{1 -n, (n-1)/2 } \Cs_{log} \vp$, which coincides with (\ref{lkut4am}).
\end {proof}

\begin{remark} The second equality in (\ref{ltqbkutm}) is in the spirit of Helgason's inversion formula in \cite[Theorem 1.17]{H11}, where the differential operator is placed outside. The first equality in (\ref{ltqbkutm}), where we first apply the differential operator and then the averaging operator, seems to be new. Formulas of both types are well known in the theory of the hyperplane Radon transforms; cf. \cite[Theorems 3.1 and 3.8]{H11}.
\end{remark}

Now consider inversion of the cosine transform   $\Cs^1 f$.

\begin{theorem}  \label{jhb67a} Let  $\vp =\Cs^1 f$, and let $f$, $n$, and $c_n$ have the same meaning as in Theorem \ref{jhb67}.

\noindent {\rm (i)} If $n$ is even, $D_1=c_n \Delta_{1-n, n/2}$, $D_2=c_n \Delta_{-1-n, n/2}$, then
\be\label{ltqbkut}
 f= FD_1\vp\quad \text {or}\quad f=D_2 F\vp.\ee

\noindent {\rm (ii)} If $n$ is odd, $ D_3= \Delta_{-1 -n, (n+1)/2 } $, then
\be\label{lkut4a} f\!=\!c\,\vp_0 +D_3\Cs_{log} \vp, \quad \vp_0\!=\!\!\intl_{\sn}\!\! \!\vp(u)\, d_*u, \quad c\!=\!\frac{\Gamma ((n\!+\!1)/2)}{\Gamma (-1/2)}.\ee
 \end{theorem}
  \begin{proof} (i) We set $\lam +2\ell=1$ in (\ref{lqbkut}) to get
  \[
 f=\Cs^{2\ell -1 -n}\Delta_{1-2\ell, \ell} \Cs^{1} f= \Cs^{2\ell -1 -n}\Delta_{1-2\ell, \ell}\, \vp,  \qquad  2\ell-1-n\neq 0,2,4, \ldots \, .  \]
  Then we choose $\ell=n/2$, which gives $  f= \Cs^{-1}\Delta_{1-n, n/2} \,\vp =FD_1 \vp.$

 If we start with  (\ref{lkyyut}) (set $\lam =1$), we obtain
 \[ f = \Delta_{-1 -n, \ell} \Cs^{-1 -n +2\ell} \vp, \qquad   -1 -n +2\ell\neq 0,2,4, \ldots \, .  \]
 Choosing  $\ell=n/2$, we get $f=\Delta_{-1-n, n/2} \Cs^{-1} \vp=D_2 F\vp$.

 (ii) If $n$ is odd,  then  $\vp-\vp_0=\Cs^1 [f - c\, \vp_0]$, and therefore,  by  (\ref{lase2hlo}),
\[
f - c \,\vp_0 =\Cs^{-1 -n} [\vp-\vp_0]=\Delta_{-2\ell, \ell}  \Cs_{log} [\vp-\vp_0], \qquad 2\ell= n+1.\]
This gives  (\ref{lkut4a}).
\end{proof}

\section{The $\lam$-Sine Transform}

The normalized $\lam$-sine transform is defined by
\be\label{tag1.10aach5}
(\S^\lam f)(u)\!=\! \del (\lam) \!\intl_{\sn} (1 \!-  \!|u\cdot v|^2)^{\lam/2} f(v) \,d_*v , \qquad u \in \sn,\ee
\[
 \del (\lam)=\frac{\pi^{1/2}\Gam (-\lam/2)} {\Gam (n/2)\, \Gam ((n-1+\lam)/2)}, \qquad Re \,\lam> 1-n, \qquad \lam \neq 0, 2, 4,\ldots \,;\]
see \cite [formula (5.1.11)]{Ru15}. Here $(1- |u\cdot v|^2)^{1/2}$ is the sine of the angle between the unit vectors $u$ and $v$.
If $f\in C^\infty_{even}(S^{n-1})$, then $\S^\lam f$ extends meromorphically to all complex $\lam$ with the only poles $\lam = 0, 2, 4, \ldots$. In particular,
\be\label{tyy5}
\S^{1-n} f \equiv \lim\limits_{\lam \to 1-n} \S^\lam f=f\ee
and
\be\label{tyy51}
\S^\lam f=c_n  \Cs^\lam F f=c_n  F\,\Cs^\lam f, \qquad c_n=\frac{\pi^{1/2}}{\Gam ((n\!-\!1)/2)}.\ee
These equalities can be found in \cite [Theorem 5.5]{Ru15}, where they have been proved  using spherical harmonics. They can  also be obtained without spherical harmonics;
cf.  \cite[Theorem 6.4 and Corollary 4.7]{Ru13}   for more general  Stiefel manifolds in different notation.

If $\int_{\sn} f(v)\, d_*v=0$, then $\int_{\sn} (Ff)(v)\, d_*v=0$, and we can pass to the limit in (\ref{tyy51}) as $\lam \to 0$. Setting
 \be\label {sin4}
(\S_{log}f)(u) \! =\!\frac{2\pi^{1/2}}{\Gam (n/2)\Gam ((n\!-\!1)/2)} \intl_{\sn} \! f(v)\, \log \frac{1}{1\!-\! |u\cdot v|^2 } \,d_*v,
\ee
we obtain
 \be\label {sin3z}
\S_{log}f  =c_n \Cs_{log}F f. \ee

\begin{lemma} \label{zret8}
Let  $f\in C^\infty_{even}$, and let $\ell$ be a nonnegative integer.

\noindent {\rm (i)} If $\lam \in \bbc$, $ \lam +2\ell\neq 0,2,4, \ldots$, then
\be\label {sin2h}
\Delta_{\lam, \ell}\, \S^{\lam +2\ell} f=\S^{\lam} f,  \ee
where  $\Delta_{\lam, \ell}$ is the  differential operator (\ref{lase1h}). In particular, if $n$ is even, then
\be\label {sin2hv}
\Delta_{1-n, \ell}\, \S^{1-n +2\ell} f=f.  \ee

\noindent {\rm (ii)} If $\int_{\sn} f(v)\, d_*v=0$,  then
 \be\label {lase2sx}
\Delta_{-2} \S_{log}f  =\S^{-2} f.\ee
\end{lemma}
\begin{proof} The first statement follows if we combine
(\ref{tyy51}) with (\ref{lase2h}). Further, by
(\ref{tyy51}) and  (\ref{lase2s}),
\[
\S^{-2} f\!=\!c_n  \Cs^{-2} F f\!=\!c_n \Delta_{-2} \Cs_{log} F f\!=\!\Delta_{-2}\S_{log}f.\]
The last equality holds by (\ref{sin3z}). This completes the proof.
\end{proof}

\section{Lower-dimensional Transforms}\label{sphere2}

The Funk transform $Ff$ and the  $\lam$-cosine transform $\Cs^\lam f$ in Section \ref{sphere1}
are associated with  cross-section of $\sn$ by  hyperplanes of codimension one. Below we consider similar transforms associated with lower-dimensional cross-sections of  codimension $k>1$. The corresponding $(n-k)$-dimensional plane in $\rn$ has an equation $u^\top x=0$, where $u$ is an $n\times k$ matrix  satisfying
$u^\top u=I_k$, $I_k$ being the identity $k\times k$ matrix, and $u^\top$  the transpose of $u$. The set of all such matrices forms
 the Stiefel manifold $\vnk$ of orthonormal $k$-frames in $\rn$.

Generalizing (\ref{af}), we introduce a dual pair of integral transforms
\be\label{afsx}    (\Cs^\lam_{k} f)(u)= \gam_k(\lam)  \intl_{\sn}  f(v)
|u^\top v|^{\lam} \,d_*v,\qquad u \in \vnk,  \ee
\be\label{afsxd}    (\stackrel{*}{\Cs}\!{}^\lam_{k} \vp)(v)= \gam_k(\lam)  \intl_{\vnk} \vp (u)
|u^\top v|^{\lam} \,d_*u,\qquad v \in \sn,  \ee
where  $|u^\top v|$ is the length of the $k$-vector $u^\top v$,
\be\label{berenj}
\gam_k(\lam)\!=\!\frac{\pi^{1/2}\,\Gamma( -\lam/2)}{\Gamma (n/2)\, \Gamma ((\lam +k)/2)}, \qquad \Re \, \lam \!>\!-k, \quad \lam
\!\neq \!0,2,4, \ldots .\ee

If $f$ and $\vp$ are smooth, then the corresponding integrals extend meromorphically to all complex $\lam \neq 0, 2, 4, \ldots \,$.
The relevant Funk type transform and its dual are defined by
\be\label{Funkkk}
 (F_{k}f)(u)=\!\!\!
 \intl_{\{v\in S^{n-1} :\,
 u^\top v =0\}} \!\!\!\! f(v) \,d_u v,\qquad u \in \vnk,  \ee
\be\label{Funkkkdu}
 ( \stackrel{*}{F}_{k} \vp)(v)=\!\!\! \intl_{\{u\in\vnk :\,
 u^\top v =0\}} \!\!\!\! \vp(u) \,d_v u,\qquad v \in \sn,
\ee
  $d_u v$ and $d_v u$ being the corresponding probability measures.

Operators (\ref{Funkkk}) and   (\ref{Funkkkdu})  actually coincide with the totally geodesic transform and its dual associated with  $(n-k-1)$-dimensional totally geodesic submanifolds of $\sn$. The latter interpretation was used, e.g., in \cite {H11, Ru02}.

 By analytic continuation,
\be\label{tnnk}
\Cs^\lam_{k} f \big |_{\lam =-k} =\mu_k F_{k} f, \qquad \stackrel{*}{\Cs}\!{}^\lam_{k} \vp \big |_{\lam =-k} =\mu_k\!\stackrel{*}{F}_{k} \vp, \ee
\[ \mu_k=\frac{\pi^{1/2}}{\Gam ((n\!-\!k)/2)}.\]

As in (\ref{tyy51}),  if $f\in C^\infty (S^{n-1})$, then for all complex $\lam$ satisfying  $\lam \neq 0, 2, 4, \ldots$
we have
\be\label{tyy51k}
\S^\lam f=c_{n,k}   \stackrel{*}{\Cs}\!{}^\lam_{k} F_{k} f= c_{n,k} \stackrel{*}{F}_{k} \! \Cs^\lam_{k} f, \quad c_{n,k}=\frac{\Gam (k/2)}{\Gam ((n-1)/2)}.\ee
In different notation,
the equalities (\ref{tnnk}) and (\ref{tyy51k})
can be found in \cite [formulas (1.8) and (1.12)]{Ru02}; see also \cite[formulas (7.10), (7.18), and Theorem 4.5]{Ru13} for similar operators on Stiefel manifolds.

 Setting $\lam =1-n$ and noting that  $\S^{1-n} f =f$ (see (\ref{tyy5})), we obtain
\be\label{tyy51ky}
c_{n,k}   \stackrel{*}{\Cs}\!{}^{1-n}_{k} F_{k} f=\S^{1-n}  f=f,
\ee
where  $\stackrel{*}{\Cs}\!{}^{1-n}_{k}$ and $\S^{1-n}$ are understood in the sense of analytic continuation.

\subsection{Inversion Formulas}

As in the case $k=1$, here we  have several options. We review only some of them and leave the rest to the interested reader.

\begin{theorem} \label{kaksa} Let  $f\in C^\infty_{even}(S^{n-1})$,  $1\le k\le n-1$,
\[
 c_{n,k}=\frac{\Gam (k/2)}{\Gam ((n-1)/2)}, \qquad \mu_k=\frac{\pi^{1/2}}{\Gam ((n\!-\!k)/2)}, \qquad c=c_{n,k}\mu_k.\]

 \noindent {\rm (i)} If $n-k$ is  odd and $g =\stackrel{*}{F}_k F_k f$,  then
\be\label {lase1hss7}
f(v)= c \,(\Delta_{1-n, \ell} \,g)(v)= c\left (-\frac{1}{4}\right )^\ell  (\Del^\ell E_{-k} \,g)(x)\Big |_{x=v},  \ee
where $ \ell=(n-k-1)/2$.

\noindent {\rm (ii)} If $n-k$ is  even, $k>1$, and $h= \stackrel{*}{\Cs}\!{}^{1-k}_k F_k f$,  then
\be\label {lbbgs}
f(v)=  c_{n,k}(\Delta_{1-n, \ell} \,h)(v)= c_{n,k} \left (-\frac{1}{4}\right )^\ell  (\Del^\ell E_{1-k +2\ell} \,h)(x)\Big |_{x=v},\ee
where  $ \ell=(n-k)/2$.
\end{theorem}
\begin{proof}   Let $\vp=F_k f$. By the first equality in (\ref{tyy51k}) and the second equality in  (\ref{tnnk}),
\be\label {flbbgs}
\S^{-k} f=c_{n,k} \stackrel{*}{\Cs}\!{}^{-k}_k \vp=c \stackrel{*}{F}_k \vp, \ee
\[     c=c_{n,k}\mu_k=\frac{\pi^{1/2}\Gam (k/2)}{\Gam ((n\!-\!1)/2)\, \Gam ((n\!-\!k)/2)}.\]
  If $n-k$ is  odd, then, by (\ref{sin2hv}),
\[
f=\Delta_{1-n, \ell}\, \S^{1-n+2\ell} f =\Delta_{1-n, \ell}\, \S^{-k} f\quad \text{\rm if $\ell=(n-k-1)/2$}.\]
Hence, by (\ref {flbbgs}), $f=\Delta_{1-n, \ell}\, \S^{-k} f= c \,\Delta_{1-n, \ell} \stackrel{*}{F}_k \vp$,
as desired.

 If $n-k$ is even, we choose  $\ell$ in (\ref{sin2hv}) so that $1-n+2\ell=1-k$  i.e. $\ell=(n-k)/2$. Then $f=\Delta_{1-n, \ell} \,\S^{1-k} f$.
Further, for $\vp=F_k f$,  the first equality (\ref{tyy51k}) with $\lam =1-k$ yields  $\S^{1-k} f=c_{n,k} \stackrel{*}{\Cs}\!{}^{1-k}_k \vp$.
Hence $f= c_{n,k}\Delta_{1-n, \ell} \,\stackrel{*}{\Cs}\!{}^{1-k}_k \vp $, which gives (\ref{lbbgs}).
\end{proof}

\begin{remark} Formula  (\ref{lbbgs}) is new.
It is inapplicable if $k=1$ because the coefficient $\gam_k (\lam) $  in  $ \stackrel{*}{\Cs}\!{}^{\lam}_k \vp$ has a pole at $\lam =0$. The  case $k=1$ with $n$ odd falls into the scope of (\ref{lkut4a}), which invokes  the logarithmic cosine transform.
\end{remark}

\begin{remark} Formula (\ref{lase1hss7}) agrees with the second equality in (\ref{ltqbkutm}) for $k=1$ and has the same structure as Helgason's inversion formula in \cite [Theorem 1.17]{H11}. Specifically, in all these  formulas  the differential operator is placed to the left of the dual transform $\stackrel{*}{F}_k$. On the other hand, Theorem \ref{jhb67} (i) reveals a remarkable intertwining property:
\be\label{plic} FD\vp=D F\vp,  \qquad \vp=Ff, \qquad D=(c_n)^2 \Delta_{1-n, (n-2)/2}.\ee
  This observation suggests the following problem  motivated by  Theorem \ref{kaksa} (i).
\end{remark}

\noindent{\bf Open Problem:} Which differential operator $\tilde D_{k}$ on the Stiefel manifold $\vnk$ is intertwined with $D_k=c\Delta_{1-n, (n-k-1)/2}$  by the dual Funk transform $\stackrel{*}{F}_k$,
i.e., 
 \be\label{plic1}
  \stackrel{*}{F}_k \tilde D_{k} \vp=D_k \stackrel{*}{F}_k \vp, \qquad \vp=F_kf, \quad \text{\rm if  $n-k$ is odd}?\ee

\subsubsection{Associated Cosine Transforms}

We restrict to inversion of   $\Cs^1_k f$.  If $n$ is even, we proceed as above. Specifically, by (\ref{sin2hv}),
\[
f=\Delta_{1-n, \ell}\, \S^{1-n+2\ell} f =\Delta_{1-n, \ell}\, \S^{1} f\quad \text{\rm if $\ell=n/2$}. \]
By the second equality in (\ref{tyy51k}) with $\lam =1$,
\be\label {lbbgs2} \S^{1} f=c_{n,k} \stackrel{*}{F}\!{}_k \Cs^1_k f, \qquad  c_{n,k}=\frac{\Gam (k/2)}{\Gam ((n-1)/2)}.\ee
Hence the desired inversion formula has the form
\be\label {lbbns}
f= c_{n,k}\Delta_{1-n, n/2}\stackrel{*}{F}\!{}_k  \Cs^1_k f,\ee
or
\[
f(v)= c_{n,k} \left (-\frac{1}{4}\right )^{n/2}  (\Del^{n/2}  E_{1 +n} \, \psi)(x)\Big |_{x=v}, \qquad \psi =\stackrel{*}{F}\!{}_k  \Cs^1_k f.\]

If $n$ is odd, the inversion formula is more complicated and we use the factorization (\ref{tyy51}) with $\lam =1$. Then, by (\ref{lbbgs2}),
\[
c_{n,k} \stackrel{*}{F}\!{}_k \Cs^1_k f=\S^{1} f=c_n  \Cs^1 F f, \qquad c_n\!=\!\frac{\pi^{1/2}}{\Gam ((n\!-\!1)/2)}.\]
The latter gives the desired inversion in the product form
\be\label {lbbns3}
f=c\,  F^{-1} (\Cs^1)^{-1}\stackrel{*}{F}\!{}_k \Cs^1_k f, \qquad c=c_{n,k}/c_n = \pi^{-1/2}\Gam (k/2),\ee
where the inverse operators  $F^{-1}$ and $(\Cs^1)^{-1}$ can be defined, e.g.,  by Theorems \ref {jhb67} and \ref{jhb67a}, respectively.

\vskip 0.3 truecm
\noindent {\bf Acknowledgement.} I would like to thank  Gestur \'Olafsson for useful discussions.


\begin{thebibliography}{10}


\bibitem {Fu11}  P. G.  Funk, \textit{\"{U}ber Fl\"{a}chen mit lauter geschlossenen geod\"{a}tischen Linien,}  Thesis, Georg-August-Universität G\"{o}ttingen, 1911.

\bibitem {Fu13}   P. G. Funk,   \textit{\"Uber Fl\"achen mit lauter geschlossenen geod\"atschen Linen},  Math. Ann. {\bf 74} (1913),  278--300.


\bibitem {Ga} R. J. Gardner,  \textit{Geometric tomography (second edition)}, Cambridge University Press, New York, 2006.


\bibitem  {GGG}  I. M. Gelfand, S. G.  Gindikin,  and  M. I. Graev,  \textit{Selected topics in integral geometry}, Translations of Mathematical Monographs, AMS, 
Providence, Rhode Island,   2003.



\bibitem {H59} S. Helgason, \textit{Differential operators on homogeneous spaces}, Acta  Math. {\bf 102} (1959),  239--299.

\bibitem {H11}  S.  Helgason,  \textit{Integral geometry and Radon transform},   Springer, New York-Dordrecht-Heidelberg-London, 2011.


\bibitem {Ko} A. Koldobsky.  \textit{Fourier analysis in convex geometry,} Mathematical Surveys and Monographs, {\bf 116}, AMS, Providence RI, 2005.

\bibitem{Lu} E. Lutwak, \textit{Centroid bodies and dual mixed volumes},  Proc. London Math.
  Soc. {\bf 60}  (1990), no. 2,  365--391.

   

\bibitem{OPR} G.{\'O}lafsson, A. Pasquale,  and  B. Rubin, \textit{Analytic and group-theoretic aspects of the cosine transform}, Contemp. Math. {\bf 598} (2013),  167--188.

\bibitem {P}  V. Palamodov,  \textit{ Reconstructive integral geometry}, Monographs in Mathematics, {\bf 98}, Basel: Birkh\"auser Verlag,  2004.

\bibitem{Ru02} B. Rubin,  \textit{Inversion formulas for the spherical {R}adon transform and the
  generalized cosine transform},   Adv. in Appl. Math. {\bf  29} (2002), no. 3,  471--497.

\bibitem{Ru13} B. Rubin,  \textit{Funk, cosine, and sine transforms on Stiefel and Grassmann
  manifolds}, J. of Geom. Anal. {\bf 23} (2013), no. 3,  1441--1497.

\bibitem{Ru15} B. Rubin, \textit{Introduction to Radon transforms (with elements of fractional calculus and harmonic analysis)}, Encyclopedia of Mathematics and Its Applications, {\bf 160}, Cambridge University Press, New York,  2015.



\bibitem  {Sa83}   S. G. Samko,  \textit{ Generalized Riesz potentials and hypersingular integrals with homogenous characteristics, their symbols and
inversion},  Proc. Steklov Inst. Math. {\bf 156} (1980),  157--222; translated in Proceeding of the Steklov Inst. of Math. 
 {\bf 2} (1983),  173--243.

\bibitem  {Sa83a}  S. G. Samko,  \textit{Singular integrals over a sphere and the construction of the characteristic from the symbol},   Izv. Vyssh. Uchebn. Zaved. Mat., no. 4 (1983), 28–42. Translated in Soviet Math. (Iz. VUZ) {\bf 27} (1983), no. 4, 35--52.

\bibitem{Se}  V. I. Semjanistyi, \textit{Some integral transformations and integral
  geometry in an elliptic space},   Trudy Sem. Vektor. Tenzor. Anal. {\bf 12}   (1963), 397--441.

\bibitem  {Zha}  Gaoyong Zhang, \textit{ Sections of convex bodies}, Amer. J. Math. {\bf 118} (1996),  319--340.

\end{thebibliography}
\end{document}